\documentclass[10pt,letterpaper]{article}
\usepackage[margin=1.7in]{geometry}
\usepackage{amssymb}
\usepackage{amsmath}
\usepackage{amsthm}
\usepackage{MnSymbol}
\usepackage{array}
\usepackage{rotating}
\usepackage{relsize}
\usepackage{verbatim}
\usepackage{xcolor}
\usepackage{extarrows}
\usepackage{enumitem} 
\usepackage{tikz-cd}
 
\theoremstyle{plain}
\newtheorem{theorem}[equation]{Theorem}
\newtheorem{corollary}[equation]{Corollary}

\newtheorem{proposition}[equation]{Proposition}

\theoremstyle{definition}

\newtheorem{definition}[equation]{Definition}
\newtheorem{example}[equation]{Example}
\newtheorem{remark}[equation]{Remark}

\makeatletter
\newcommand*{\da@rightarrow}{\mathchar"0\hexnumber@\symAMSa 4B }
\newcommand*{\da@leftarrow}{\mathchar"0\hexnumber@\symAMSa 4C }
\newcommand*{\da@xarrow}[7]{%
  \sbox0{$\ifx#7\scriptstyle\scriptscriptstyle\else\scriptstyle\fi#5#1#6\m@th$}%
  \sbox2{$\ifx#7\scriptstyle\scriptscriptstyle\else\scriptstyle\fi#5#2#6\m@th$}%
  \sbox4{$#7\dabar@\m@th$}%
  \dimen@=\wd0 %
  \ifdim\wd2 >\dimen@
    \dimen@=\wd2 %
  \fi
  \count@=2 %
  \def\da@bars{\dabar@\dabar@}%
  \@whiledim\count@\wd4<\dimen@\do{%
    \advance\count@\@ne
    \expandafter\def\expandafter\da@bars\expandafter{%
      \da@bars
      \dabar@
    }%
  }%
  \mathrel{#3}%
  \mathrel{%
    \mathop{\da@bars}\limits
    \ifx\\#1\\%
    \else
      _{\copy0}%
    \fi
    \ifx\\#2\\%
    \else
      ^{\copy2}%
    \fi
  }%
  \mathrel{#4}%
}
\makeatother

\DeclareMathOperator\Lie{Lie}

\DeclareMathOperator\Ad{Ad}

\DeclareMathOperator\Aut{Aut}
\DeclareMathOperator{\character}{char}

\DeclareMathOperator\Spin{Spin}

\DeclareMathOperator\Dynkin{Dynkin}

\DeclareMathOperator\Int{Int}
\DeclareMathOperator\Out{Out}

\newcommand{\fin}{\text{\textnormal{fin}}}

\begin{document}
\title{Disconnected reductive groups}
\author{Marisa Gaetz\thanks{Marisa Gaetz was supported by the NSF Graduate Research Fellowship Program under Grant No.~2141064 and by the Fannie \& John Hertz Foundation.} \\Department of
  Mathematics \\ MIT, Cambridge, MA 02139
\and David A. Vogan, Jr.\\2-355, Department of Mathematics\\ MIT,
  Cambridge, MA 02139}

\date{}
\maketitle

\begin{abstract}
In this paper, we describe the possible disconnected complex reductive algebraic groups $E$ with component group $\Gamma = E/E_0$. We show that there is a natural bijection between such groups $E$ and algebraic extensions of $\Gamma$ by $Z(E_0)$. 
\end{abstract}

\section{Introduction}\label{sec:intro}
\setcounter{equation}{0}
This paper is concerned with the general problem of understanding and classifying \textit{disconnected} reductive algebraic groups. If $E$ is such a group, and $G:= E_0$ its identity
component, then
\begin{equation}\label{eq:disc1}\begin{aligned}
    \Gamma := E/&E_0 = E/G\\
  1 \longrightarrow G \longrightarrow &E \buildrel p_E \over
  \longrightarrow \Gamma \longrightarrow 1 
\end{aligned}\end{equation}
with $\Gamma$ a finite group. We will begin in Section \ref{sec:connected} with Chevalley's description of the connected reductive algebraic group $G$. We will then use this to describe the possibilities for $E$ in Section \ref{sec:disconnected}. The main result of the paper, formulated more completely in Theorem \ref{thm:listE}, is as follows:

\begin{theorem} \label{thm:main}
Let $\Gamma$ be a finite group. Let $G$ be a connected reductive algebraic group over a field $k$ with $\character (k) \nmid \vert \Gamma \vert$, equipped with a group homomorphism $\Ad \colon \Gamma \rightarrow \Out (G)$. Then the isomorphism classes of algebraic extensions $E$ of $\Gamma$ by $G$ such that the induced map $\Gamma \rightarrow \Out(G)$ agrees with $\Ad$ are parameterized by 
$$H^2(\Gamma, Z(G)) \simeq H^{2} (\Gamma, Z(G)_{\fin} ),$$
where $Z(G)_{\fin }$ is the torsion subgroup of the center of $G$.
\end{theorem}

\begin{example}
Suppose that $G = GL(n,\mathbb{C})$ and that $\Gamma = \mathbb{Z}/r\mathbb{Z}$ for some $n,r \in \mathbb{N}$, with $\Gamma$ acting trivially on $G$. In this case, $H^2(\Gamma, Z(G)) = H^2(\mathbb{Z}/r\mathbb{Z}, \mathbb{C}^{\times}) = 0$. Therefore, Theorem \ref{thm:main} implies that any extension of $\Gamma$ by $G$ is a semidirect product.
\end{example}

\begin{example}
Suppose that $G = SL(2n,\mathbb{C})$ (with $n\ge 2$), and that $\Gamma = {\mathbb Z}/2{\mathbb Z}$ acts by outer automorphisms on $G$. This means that some element of the non-identity component of an extension $E$ acts on $G$ by $g\mapsto {}^t g^{-1}.$ In this case $Z(G) = \mu_{2n}$, the group of $2n$-th roots of unity, and $\Gamma$ acts on $Z(G)$ by inversion. (The reason is that $Z(G)$ consists of scalar matrices, and inverse transpose sends such a matrix to its inverse.) It is easy to calculate
$$H^2(\Gamma, Z(G)) = (1,\text{(one other class)}).$$
The first of the corresponding extensions of $\Gamma$ by $Z(G)$ is the semidirect product
$$Z(G) \rtimes \Gamma.$$
This is the dihedral group, generated by a primitive $2n$-th root of 1 (which we will call $\omega$), and an element $\delta$ satisfying
$$\delta\omega \delta^{-1} = \omega^{-1}, \quad \delta^2 = 1.$$
The dihedral group has $2n + 1$ elements of order $2$: the $2n$ elements $z\delta$ (with $z\in \mu_{2n}$), and $-1 \in \mu_{2n}$.

The second extension is generated by $\omega$ and an element $\delta'$ satisfying
$$\delta'\omega (\delta')^{-1} = \omega^{-1}, \quad (\delta')^2 = \omega^n = -1.$$
In the case $n=2$, this is the quaternion group of order eight. This second extension has just one element of order 2, so is not isomorphic to the dihedral group.

To make the corresponding extensions of $G = SL(2n,{\mathbb C})$, it seems natural to adjoin to $G$ an element $x$ satisfying
$$ xgx^{-1} = {}^t g ^{-1}, \quad x^2 = z \in Z(G).$$
These relations are consistent as long $z = \pm 1$; but they are {\em not} the relations we are going to use. The reason is that we will be taking advantage of the precise structure theory for $G$ created by Chevalley. In the case of $SL(n,{\mathbb C})$, this means using the Borel subgroup of upper triangular matrices; and the inverse transpose automorphism takes upper triangular matrices to lower triangular matrices. We therefore define $J$ to be the square antidiagonal matrix of size $2n$, with entries
$$J_{pq} = \delta_{p+q, 2n+1}.$$
The automorphism of $SL(2n,{\mathbb C})$
$$ d(g) = J\,{}^t g^{-1}J^{-1}$$
{\em does} preserve upper-triangular matrices. We can form extensions
$E_{\pm 1}$ generated by $SL(2n,{\mathbb C})$ and an element $\delta$ satisfying
$$\delta g\delta^{-1} = d(g), \quad \delta^2 = \pm 1.$$
These two extensions $E_{\pm 1}$ are the two extensions of $\Gamma$ by $SL(2n,{\mathbb C})$ described by the theorem.

The case of $SL(2n+1,{\mathbb C})$ with a nontrivial $\Gamma$ action can be analyzed in the same way; in this case the conclusion is that $H^2(\Gamma, Z(G))$ is trivial, so any extension must be a semidirect product.
\end{example}

\begin{example}
Suppose $V = ({\mathbb Z}/2{\mathbb Z})^2$ is the Klein four-group, and $\Gamma = S_3$ acts on $V$ by permuting the three elements of order two. It is well known that there is a short exact sequence
$$1\rightarrow V \rightarrow S_4 \rightarrow S_3 \rightarrow 1,$$
with $V$ identified with the group of even permutations of order one or two in $S_4$. It is also well known that this extension is a semidirect product (any standard inclusion of $S_3$ in $S_4$ providing a section). It is probably well known that 
$$H^2(S_3,V)$$
is trivial, so that $S_4$ is the {\em only} extension; but this was not known to us, so we will provide a proof below.

It now follows from our main theorem that {\em any} extension of $\Spin(8,{\mathbb C})$ by $S_3$ (the outer automorphism group of $D_4$) must also be trivial. We will describe such an extension constructed in an interesting way, for which the triviality is not evident.

Let $H$ be a simple complex reductive group of type $F_4$, $T$ a maximal torus, and $G \supset T$ the subgroup corresponding to the long roots. Then it is well known that 
$$G \simeq \Spin(8,{\mathbb C}), \quad Z(G) \simeq V.$$
Because the Weyl group $W(H,T)$ must preserve the set of long roots, the normalizer of $T$ in $H$ must be contained in the normalizer of $G$ in $H$. It follows easily that
$$N(G,H)/G \simeq N(T,H)/N(T,G) \simeq W(F_4)/W(D_4).$$
These last two groups have orders 1152 and 192 respectively, and it follows easily that
$$W(F_4)/W(D_4) \simeq S_3,$$
acting on $D_4$ by diagram automorphisms. We now have an extension
$$E = N(\Spin(8,{\mathbb C}),H),$$
$$1 \rightarrow \Spin(8,{\mathbb C}) \rightarrow E \rightarrow S_3 \rightarrow 1,$$
with $S_3$ acting on $\Spin(8,{\mathbb C})$ by diagram automorphisms. 

Here is a proof of the triviality of $H^2(S_3,V)$. Suppose we have a group extension
$$1\rightarrow V \rightarrow F \rightarrow S_3 \rightarrow 1,$$
with $S_3$ acting on $V$ by permuting the nontrivial elements $\{\epsilon_i\mid i\le i\le 3\}$. We want to prove that $F$ must be a semidirect product. This means that we must choose preimages for generators of $S_3$ that multiply in the same way as the elements of $S_3$. Now $S_3$ is generated by two transpositions $s$ and $t$, subject to the relations
$$s^2 = t^2 = 1, \quad sts = tst.$$
Consider a transposition $s_{ij} \in S_3$. Let $\tilde {s_{ij}}$ be any preimage in $E$. Then $\tilde{s_{ij}}^2$ must be in $V$, and must be fixed by the transposition $s_{ij}$; so 
$$\tilde{s_{ij}}^2 = 1 \ \text{or} \ \epsilon_k,$$
with $k$ the third element of $\{1,2,3\}$ (not equal to $i$ or $j$). In the first case we are happy. In the second case we can replace our representative by
$$\tilde{s_{ij}}' = \tilde{s_{ij}} \epsilon_i,$$
and calculate
$$(\tilde{s_{ij}}')^2 = \epsilon_k\epsilon_i\epsilon_j =1.$$

Choose now two generating transpositions $s$ and $t$ for $S_3$, and choose representatives $\tilde s$ and $\tilde t$ of order 2. As representative of the three-cycle $st$, we choose $\widetilde{st} =\tilde s\tilde t$. The cube of this element must be a representative of $1\in S_3$, so must belong to $V$; and clearly it must be fixed by the 3-cycle $st$. The only such element of $V$ is the identity, so
$$(\tilde s \tilde t)^3 = 1,$$
or equivalently
$$\tilde s\tilde t \tilde s = \tilde t \tilde s \tilde t.$$
This is the braid relation for $S_3.$ Our (slightly modified) representatives $\tilde s$ and $\tilde t$ therefore satisfy the defining relations of $S_3$, and provide the cross section proving that $E$ must be a semidirect product, as we wished to show.
\end{example}

\section{Connected reductive algebraic groups} \label{sec:connected}
\setcounter{equation}{0}

We begin with a complex connected reductive algebraic group $G$ equipped
with a \textit{pinning}:
\begin{definition}[see for example {\cite[Definition 1.18]{AV}}]
Suppose $G$ is a complex connected reductive algebraic group. A \textit{pinning} of $G$ consists of
\begin{enumerate}
    \item a Borel subgroup $B \subset G$;
    \item a maximal torus $H \subset B$; and
    \item for each simple root $\alpha \in \Pi (B,H)$, a choice of basis vector $X_{\alpha} \in \mathfrak{g}_{\alpha}$.
\end{enumerate}
\end{definition}
\begin{subequations}\label{se:rootdata}
The reason we use \cite{AV} as a reference rather than something older and closer
to original sources is that this reference is concerned, as we will
be, with representation theory of disconnected groups. Write
\begin{equation} \label{eq:roots}
\begin{tabular}{r c l c r c l }
    $R(G,H)$ & $\subset$ & $X^*(H)$, & \hspace{.25cm} & $R^{\vee} (G,H)$ & $\subset$ & $X_* (H)$  \\
    $\Pi (B,H)$ & $\subset$ & $R^+ (B,H)$, & \hspace{.25cm} & $\Pi^{\vee}(B,H)$ & $\subset$ & $(R^{\vee})^+ (B,H)$ 
\end{tabular}
\end{equation}
for the roots, coroots, simple roots, and simple coroots that are specified by the pinning $\{ B,H,X_{\alpha} \}$ of $G$. Here, the (co)roots live in the (co)character lattice and the simple (co)roots live in the set of positive (co)roots. For clarity, we will try to write $\alpha$ for simple roots and $\beta$ for arbitrary roots.

Each $\beta\in R(G,H)$ defines a reflection
\begin{equation}\label{eq:sbeta}
  s_\beta \in \Aut(X^*(H)),\quad s_\beta(\lambda) = \lambda -
\langle\beta^\vee,\lambda\rangle\beta \qquad (\lambda \in X^*(H)),
\end{equation}
where $\langle \cdot , \cdot \rangle$ denotes the natural pairing between the dual lattices $X_*(H)$ and $X^*(H)$. These reflections generate the {\em Weyl group of $H$ in $G$}: 
\begin{equation}\label{eq:W}
   W(G,H):=\langle \{s_\beta \mid \beta\in R(G,H)\} \rangle \subset
    \Aut(X^*(H)).
\end{equation}
For each $s_{\beta} \in \Aut(X^*(H))$, there is a transpose automorphism of $X_* (H)$:
\begin{equation}\label{eq:sbetavee}
  s_{\beta^\vee} \in \Aut(X_*(H)),\quad s_{\beta^\vee}(\ell) = \ell -
\langle\ell,\beta\rangle\beta^\vee \qquad (\ell \in X_*(H)).
\end{equation}
The inverse transpose isomorphism
\begin{equation} \label{eq:inv-transpose}
\Aut(X^*(H)) \simeq \Aut(X_*(H)), \qquad T \mapsto {}^t T^{-1}
\end{equation}
identifies $W(G,H)$ with the group of automorphisms of $X_*(H)$
generated by the various $s_{\beta^\vee}$. 
\end{subequations}

\begin{proposition}[Springer {\cite[Proposition 8.1.1]{Springer}}] \label{prop:8.1.1}
Let $G$ be a complex connected reductive algebraic group with pinning $\{ B,H,X_{\alpha} \}$. Let $\mathbb{G}_a$ be $\mathbb{C}$, viewed as an additive group, with Lie algebra $\Lie(\mathbb{G}_a) = \mathbb C$.
\begin{enumerate}[label=(\roman*)]
    \item For $\beta \in R(G,H)$, there is an isomorphism $u_{\beta}$ of $\mathbb{G}_a$ onto a closed subgroup $U_{\beta}$ of $G$ such that $\Lie (U_{\beta}) = \mathfrak{g}_{\beta}$.
    \item If $\alpha \in \Pi(G,H)$, then $u_\alpha$ is characterized by requiring $du_\alpha(1) = X_\alpha$.
    \item $H$ and the $U_{\beta}$ with $\beta \in R(G,H)$ generate $G$.
\end{enumerate}
\end{proposition}

\begin{proposition}[Springer {\cite[Proposition 8.2.4 and Corollary 8.2.10]{Springer}}] \label{prop:8.2.4}
Let $G$ be a complex connected reductive algebraic group with pinning $\{ B,H,X_{\alpha} \}$. Let $\widetilde{R}^+$ be an arbitrary system of positive roots in $R(G,H)$, with simple roots $\widetilde\Pi$.
\begin{enumerate}[label=(\roman*)]
    \item $H$ and the $U_{\alpha}$ with $\alpha \in \widetilde\Pi$ generate a Borel subgroup of $G$.
    \item There is a unique $w \in W(G,H)$ with $\widetilde{R}^+ = w.R^+(B,H)$.
\end{enumerate}
\end{proposition}

\begin{subequations}
    
In this way, we see that the Borel subgroup $B$ in the pinning $\{ B,H,X_{\alpha} \}$ is determined by the $X_{\alpha}$, so we will often use $\{ H, X_{\alpha} \}$ to denote a pinning. 

As for any group, there are natural short exact sequences
\begin{equation}\label{eq:intout}\begin{aligned}
  1 &\longrightarrow \Int(G) \longrightarrow \Aut(G) \buildrel
  p_{\Aut}\over\longrightarrow \Out(G) \longrightarrow 1\\
  & 1 \longrightarrow Z(G) \longrightarrow G \buildrel
  p_G\over\longrightarrow \Int(G) \longrightarrow 1,
\end{aligned}\end{equation}
where $\Aut (G)$ denotes the algebraic automorphisms of $G$;
\begin{equation} \label{eq:int}
    \Int (G) := \{ \Ad (g) \mid g \in G \} = \{ \text{inner automorphisms of } G \};
\end{equation}
and $\Out (G) := \Aut (G) / \Int (G)$. In particular, $\Int (G) \simeq G/Z(G)$.

\end{subequations} 

\begin{proposition} \label{prop:pinned}
Let $G$ be a complex connected reductive algebraic group. The group $\Int (G)$ acts simply transitively on the set of pinnings for $G$.
\end{proposition}

\begin{proof}
Let $\{ B, H, X_{\alpha} \}$ and $\{ B', H', X_{\alpha'}' \}$ be two pinnings for $G$, and let $\Pi$ and $\Pi'$ denote the corresponding sets of simple roots. Then by \cite[Theorem 6.2.7]{Springer}, $B = gB'g^{-1}$ for some $g \in G$. Additionally, $gH'g^{-1}$ is a maximal torus of $gB'g^{-1} = B$, so we have that $H = bgH'g^{-1}b^{-1}$ for some $b \in B$ by \cite[Theorem 6.4.1]{Springer}. Setting $\{ X_{\alpha''}'' \} := \Ad (bg)(\{ X_{\alpha'}' \})$ (and letting $\Pi''$ denote the corresponding set of simple roots), we see that 
$$\Ad (bg)( \{ B', H', X_{\alpha'}' \} ) = \{ B, H, X_{\alpha''}'' \}.$$ 
Since a unique set of simple roots is determined by the pair $(B,H)$, we get that $\Pi = \Pi''$. Now, we claim that there exists $e^Y \in H$ such that $\Ad (e^Y) ( X_{\alpha}'' ) = X_{\alpha}$ for all $\alpha \in \Pi$. Since $X_\alpha , X_{\alpha}'' \in \mathfrak{g}_{\alpha}$, we have that $X_{\alpha} = c_\alpha X_\alpha''$ for some scalar $c_\alpha$. Therefore, our desired $e^Y \in H$ is a solution to the following system of equations:
$$ \Ad (e^Y) (X_\alpha'') = e^{\alpha (Y)} X_{\alpha}'' = c_{\alpha} X_{\alpha}'' \hspace{.5cm} \text{ for all } \alpha \in \Pi .$$
Since the simple roots are linearly independent (see \cite[Theorem 8.2.8]{Springer}), this system has a solution $h:=e^Y \in H$. Then $\Ad(hbg)(\{ H',X_{\alpha'}' \}) = \{ H,X_{\alpha} \}$, which proves that the action of $\Int (G)$ is transitive.

To see that this action is \textit{simply} transitive, suppose that 
$$\Ad(g)(\{ B,H,X_{\alpha }\}) = \Ad (g')(\{ B, H,X_{\alpha} \})$$
for some $g,g' \in G$. Then
$$(g')^{-1} g \in N_G (H) \hspace{.5cm} \text{ and } \hspace{.5cm} \Ad ((g')^{-1}g)(\{ X_{\alpha} \}) = \{ X_{\alpha} \},$$ 
so Proposition \ref{prop:8.2.4} gives that $(g')^{-1}g \cdot H = 1_{W(G,H)}$. In other words, $(g')^{-1}g \in H$. But then Proposition \ref{prop:8.1.1} and the relation 
$$\Ad ((g')^{-1}g)(\{ X_{\alpha} \}) = \{ X_{\alpha} \}$$ 
show that $\Ad ((g')^{-1}g)(X_{\alpha}) = X_{\alpha}$ for all $\alpha \in \Pi (B,H)$. It follows that $(g')^{-1}g \in Z(G)$, completing the proof.
\end{proof}

\begin{subequations} \label{se:disc1}

The {\em root datum of $G$} is the quadruple
\begin{equation}\label{eq:rootdata2}
{\mathcal R}(G) := (X^*(H),R(G,H),X_*(H),R^\vee(G,H))
\end{equation}
and the {\em based root datum} is
\begin{equation}\label{eq:basedrootdata}
{\mathcal B}(G) := (X^*(H),\Pi(B,H),X_*(H),\Pi^\vee(B,H)).
\end{equation}
It is worth noting that Proposition \ref{prop:pinned} is essentially Chevalley's theorem that every reductive algebraic group has an associated root datum, and the root datum determines the group up to isomorphism \cite{Chev}.

With these notions established, define 
\begin{equation}\label{eq:intout0}
    \Aut({\mathcal B}(G)) := \left\{T\in \Aut(X^*(H)) \mid \parbox{.35\textwidth}{$\ \ T(\Pi(B,H)) =
  \Pi(B,H)$\\$ \ {}^t T(\Pi^\vee(B,H)) = \Pi^\vee(B,H)$} \right\}   
\end{equation}
and
\begin{equation}
        \Aut(G,\{H,X_\alpha\}) := \{\tau\in \Aut(G) \mid \tau(\{H,
    X_\alpha\}) = \{H,X_\alpha \} \},
\end{equation}
\end{subequations} 
where by $\tau(\{H, X_\alpha\}) = \{H,X_\alpha \}$, we mean that $\tau (H) = H$ and that $\tau$ preserves the set $\{ X_{\alpha} \}$ (not necessarily pointwise). Then Proposition \ref{prop:pinned} implies the following:

\begin{corollary} 
  Let $G$ be a complex connected reductive algebraic group with pinning $\{ H,X_{\alpha} \}$. An automorphism of $G$ preserving the pinning is precisely the same thing as an automorphism of the based root datum:
$$\Aut ({\mathcal B}(G)) = \Aut (G, \{ H,X_{\alpha} \}) .$$
\end{corollary}

Elements of $\Aut ( {\mathcal B}(G) ) = \Aut (G, \{ H, X_{\alpha} \})$ are called \textit{distinguished} automorphisms of $G$.

\begin{corollary} \label{cor:aut=int-dist}
Let $G$ be a complex connected reductive algebraic group with pinning $\{ H,X_{\alpha} \}$. The group of algebraic automorphisms of $G$ is the semidirect product of the inner automorphisms and the distinguished automorphisms:
  $$\Aut(G) = \Int(G) \rtimes \Aut(G,\{H,X_\alpha\}).$$
  Consequently,
  $$ \Out(G) \simeq \Aut(G,\{H,X_\alpha\}) = \Aut({\mathcal B}(G)) .$$
\end{corollary}

\begin{proof}
Let $\{ B,H,X_{\alpha} \}$ be a pinning of $G$. Proposition \ref{prop:pinned} implies that 
$$\Int (G)\, \cap \,\Aut (G , \{ H,X_{\alpha} \}) = \{ 1 \}.$$ 
To see that $\Int (G)$ and $\Aut (G, \{ H,X_{\alpha} \})$ generate $\Aut (G)$, we show that there is a surjective homomorphism $\Aut (G) \rightarrow \Aut (G , \{ H,X_{\alpha} \})$ with kernel $\Int (G)$. To this end, let $\tau \in \Aut (G)$. Then $\{ \tau(B), \tau(H), \tau(X_\alpha) \}$ is a pinning of $G$, where $\tau (X_\alpha) := d(\tau \circ u_\alpha)(1)$. Therefore, by Proposition \ref{prop:pinned}, there exists a unique element $\Ad(g_{\tau}) \in \Int(G)$ such that $\Ad(g_{\tau})( \tau (H)) = H$, $\Ad(g_{\tau})( \tau (B)) = B$, and $\Ad(g_\tau)( \tau (X_\alpha)) = X_\alpha$ for all $\alpha \in \Pi (B,H)$. Then $\Ad (g_{\tau}) \circ \tau \in \Aut (G, \{ H,X_{\alpha} \})$, and we have a well-defined homomorphism
\begin{align*}
    \Aut (G) & \rightarrow \Aut (G,\{ H,X_{\alpha} \}) \\
    \tau & \mapsto \Ad (g_{\tau}) \circ \tau .
\end{align*}
It is straightforward to check that this homomorphism has kernel $\Int (G)$, as desired.
\end{proof}

\section{Disconnected reductive algebraic groups} \label{sec:disconnected}
\setcounter{equation}{0}

\begin{subequations}\label{se:disc2}
We can now describe the possible disconnected groups $E$ as in
\eqref{eq:disc1}. We will take as given the connected complex reductive
algebraic group $G$ as in Section \ref{sec:connected}, specified by the based root
datum ${\mathcal B}(G)$. We fix also a finite group $\Gamma$, which
may be specified in any convenient fashion. If we are to have a short
exact sequence as in \eqref{eq:disc1}, we will get automatically a
group homomorphism
\begin{equation}\label{eq:outer}
  \Ad \colon \Gamma \rightarrow \Out(G) \simeq \Aut (G, \{ H,X_{\alpha} \}) = \Aut({\mathcal B}(G)).
\end{equation}
We will take the specification of such a homomorphism $\Ad$
as part of the data (along with $G$ and $\Gamma$) which we are given.

For any group $G$, the inner automorphisms act trivially on the center
$Z(G)$, so there is a natural homomorphism
\begin{equation}
  \Ad\colon \Out(G) \rightarrow \Aut(Z(G)).
\end{equation}
In our setting, \eqref{eq:outer} therefore gives
\begin{equation}\label{eq:outerZ}
  \overline{\Ad}\colon \Gamma \rightarrow \Aut(Z(G)).
\end{equation}

We pause briefly to recall what sort of group is $\Aut({\mathcal
  B}(G))$. Recall that the set of simple roots $\Pi(B,H)$ is the set
of vertices of a graph $\Dynkin(G)$ with some directed multiple edges, the
{\em Dynkin diagram} of $G$: an edge joins distinct vertices $\alpha$
and $\alpha'$ if and only if $\langle (\alpha')^\vee,\alpha\rangle \ne 0$. 
We will not make further use of the Dynkin diagram, so we do not recall the details. If $Z(G)$
has dimension $m$, then
\begin{equation} \label{eq:out}
  \Aut({\mathcal B}(G)) \subset \Aut(\Dynkin(G)) \times \Aut({\mathbb
    Z}^m).
\end{equation}
Here the first factor is a finite group of ``diagram automorphisms'' and the second is the discrete group of $m\times m$ integer matrices
of determinant $\pm 1$. To see why we have this containment, recall that $\Aut ( {\mathcal B} (G) )$ is the set of automorphisms of the character lattice $X^* (H)$ preserving the set of simple roots. The character lattice $X^*(H)$ has a sublattice of finite index
\begin{equation}\label{eq:almost_product}
    X^* (H) \supset X^*( H/Z(G)) \times X^* (H/(H\cap [G,G]))  \simeq \mathbb{Z}^{\dim (H) - m} \times \mathbb{Z}^{m} 
\end{equation}
which must be preserved by any automorphism of the root datum. The automorphisms of $X^*(H/Z(G))$ preserving the set of simple roots are precisely the diagram automorphisms. If $L_0 \subset L$ is a sublattice of finite index, restriction to $L_0$ defines an inclusion
$$
\Aut(L) \hookrightarrow \Aut(L_0),
$$
identifying $\Aut(L)$ as a subgroup of finite index in $\Aut(L_0)$. It follows that the inclusion \eqref{eq:out} induced by \eqref{eq:almost_product} is of finite index.

The automorphism group of a connected Dynkin diagram is small (order
one or two except in the case of $D_4$, where the automorphism group
is $S_3$). So for semisimple $G$ the possibilities for the
homomorphism $\Ad$ are generally quite limited, and can be described
concretely and explicitly. Understanding possible maps to
$\Aut({\mathbb Z}^m)$ means understanding $m$-dimensional
representations of $\Gamma$ over ${\mathbb Z}$. This is a more subtle
and complicated subject; we will not concern ourselves with it,
merely taking $\Ad$ as given somehow.

With this information in hand, the quotient group $E/Z(G)$ can now be
completely described. Define (using $p_{\Aut}$ from \eqref{eq:intout}
and $\Ad$ from \eqref{eq:outer})
\begin{equation}
  \Aut_{\Gamma}(G) := p_{\Aut}^{-1}(\Ad(\Gamma)).
\end{equation}
We immediately get from \eqref{eq:intout} a short exact sequence
\begin{equation}
  1 \longrightarrow \Int(G) \longrightarrow \Aut_\Gamma(G)
  \longrightarrow \Ad(\Gamma) \longrightarrow 1.
  \end{equation}
Moreover, Corollary \ref{cor:aut=int-dist} implies that 
\begin{equation} \label{eq:Aut_Gamma-semidirect}
    \Aut_{\Gamma}(G) = \Int (G) \rtimes \Ad (\Gamma).
\end{equation}
The identity component $G$ is a distinguished subgroup of $E$ (i.e.,~is preserved by any automorphism), so we have a map $\Ad \colon E \rightarrow \Aut (G)$. Note that 
\begin{equation} \label{eq:Ad(E)}
    \Ad (E) = \Aut_{\Gamma} (G).
\end{equation}
Now, for a pinning $\{ H,X_{\alpha} \}$ of $G$, define 
\begin{equation}
    E (\{ H, X_{\alpha} \}) := \{ e \in E \mid \Ad (e) ( \{ H,X_{\alpha} \}) = \{ H,X_{\alpha} \} \},
\end{equation}
the subgroup of $E$ defining distinguished automorphisms of $G$ (see
  \eqref{eq:intout0}).
\end{subequations}

\begin{proposition} \label{prop:E(H,X)/Z(G)->Gamma}
Let $E$ be a complex disconnected reductive algebraic group with identity component $G$. Let $\{ H,X_{\alpha} \}$ be a pinning of $G$. 
\begin{enumerate}[label=(\roman*)]
    \item $G \cap E(\{ H,X_{\alpha} \}) = Z(G)$.
    \item The map $\overline{p}_E \colon E(\{ H,X_{\alpha} \})/Z(G) \rightarrow \Gamma$ is an isomorphism.
\end{enumerate}
\end{proposition}

\begin{proof}
The inclusion $Z(G) \subseteq G \cap E(\{ H,X_{\alpha} \})$ is clear. For the other, suppose that $g \in G \cap E(\{ H,X_{\alpha} \})$. Then 
$$\Ad (g) \in \Int (g) \cap \Aut (G, \{ H,X_{\alpha} \}) = \{ 1 \},$$
where we have used Proposition \ref{prop:pinned}. It follows that $g \in Z(G)$, which proves (i).

Next, we claim that $\overline{p}_E$ is injective. To this end, let $e \in E(\{ H,X_{\alpha} \})$ and suppose that $\overline{p}_E ( e Z(G) ) = \text{id}_{\Gamma}$ (i.e.,~that $e  G = \text{id}_{\Gamma}$). Then 
$$e \in G \cap E(\{ H,X_{\alpha} \}) = Z(G),$$ 
and injectivity follows. For surjectivity, note that
\begin{align*}
    \Ad (E (\{ H,X_{\alpha} \}) ) &= \Ad (E) \cap \Aut (G,\{ H,X_{\alpha} \}) \\
    &= [ \Int (G) \rtimes \Ad (\Gamma ) ] \cap \Aut (G, \{ H,X_{\alpha} \}) \\
    &= \Ad (\Gamma),
\end{align*}
where we have used \eqref{eq:Aut_Gamma-semidirect} and Proposition \ref{prop:pinned}. Therefore, given a coset $eG \in \Gamma$, we can write $\Ad (eG) = \Ad (e')$ for some $e' \in E(\{ H,X_{\alpha} \})$. Explicitly, this means that $p_{\Aut }(\Ad (e)) = \Ad (e')$, which gives that $\Ad (e)^{-1} \circ \Ad (e') \in \Int (G)$, and hence that $e' \in eG$. In particular, 
$$ \overline{p}_E ( e' Z(G) ) = e G \in \Gamma ,$$
which proves (ii).
\end{proof}

Proposition \ref{prop:E(H,X)/Z(G)->Gamma} immediately implies the following:

\begin{corollary}
Let $E$ be a complex disconnected reductive algebraic group with identity component $G$. Then
$$E/Z(G) \simeq [G/Z(G)] \rtimes \Gamma ,$$
and $E(\{ H,X_{\alpha} \})$ is an extension of $\Gamma$ by the abelian group $Z(G)$:
\begin{equation} \label{eq:E(H,X)-SES}
    1 \longrightarrow Z(G) \longrightarrow E(\{ H,X_{\alpha} \}) \longrightarrow \Gamma \longrightarrow 1.
\end{equation}
\end{corollary}

With this established, we can now summarize the relationship between many of the short exact sequences in this paper; in the following diagram, the row and column labels indicate where the corresponding sequences and maps appear in the paper:

\vspace{1.5em}

\begin{center}
\begin{tikzcd}
    1 \arrow[r] & Z(G) \arrow[r,hook]  \arrow[d, hook] & E(\{H, X_\alpha \}) \arrow[r,two heads]  \arrow[d, hook] & \Gamma \arrow[r] \arrow[d,equal] & 1 & \hspace{-.75cm}\text{(3.2)} \\
    1 \arrow[r] & G \arrow[r,hook] \arrow[d,two heads,"p_G"] & E \arrow[r,two heads, "p_E"] \arrow[d,"\Ad"] & \Gamma \arrow[r] \arrow[d,"\Ad"] & 1 & \hspace{-.75cm} \text{(1.1)} \\
    1 \arrow[r] & \Int(G) \arrow[r,hook] \arrow[d,equal] & \Aut_{\Gamma} (G) \arrow[r,two heads] \arrow[d,hook] & \Ad (\Gamma) \arrow[r] \arrow[d,hook] & 1 & \hspace{-.75cm} \text{(3.1g)} \\
    1 \arrow[r] & \Int(G) \arrow[r,hook]  & \Aut (G) \arrow[r,two heads,"p_{\Aut}"] & \Out(G) \arrow[r] & 1 & \hspace{-.75cm} \text{(2.2a)} \\ [-15pt]
    & \text{(2.2a)} & \text{(3.1i)} & \text{(3.1a)} & &  \\
\end{tikzcd}
\end{center}

We are now ready to prove the main result of the paper, Theorem \ref{thm:listE} (summarized above in Theorem \ref{thm:main}). This theorem shows that the problem of understanding extensions of $\Gamma$ by $G$ can be reduced to the problem of understanding the extensions \eqref{eq:E(H,X)-SES} of $\Gamma$ by $Z(G)$.

\begin{theorem}\label{thm:listE}
Let $E$ be a complex disconnected reductive algebraic group with identity component $G$. Let $\{ H,X_{\alpha} \}$ be a pinning of $G$.
  \begin{enumerate}[label=(\roman*)]
  \item There is a natural surjective homomorphism
    $$ G\rtimes E(\{H,X_\alpha\}) \rightarrow E$$
      with kernel the antidiagonal copy
      $$Z(G)_{-\Delta} = \{(z,z^{-1})\mid z \in Z(G)\}$$
    of $Z(G)$. Consequently, there is a natural bijection between algebraic extensions $E$
        of $\Gamma$ by $G$ and algebraic extensions $E(\{H,X_\alpha\})$
        of $\Gamma$ by $Z(G)$ (with the action \eqref{eq:outerZ} of $\Gamma$ on $Z(G)$. These latter are parametrized by the group cohomology
        $$H^2(\Gamma,Z(G))$$
 (see for example \cite[pages 299--303]{CE}). The bijection is
        given by
        $$ E = \left[G\rtimes E(\{H,X_\alpha\})\right]/Z(G)_{-\Delta}.$$
      \item Define
        $$Z(G)_\fin = \text{elements of finite order in $Z(G)$,}$$
        the torsion subgroup. Then the natural map
        $$H^p(\Gamma,Z(G)_\fin) \rightarrow H^p(\Gamma,Z(G)) \qquad (p\ge 2)$$
        is an isomorphism.
  \end{enumerate}
\end{theorem}

\begin{proof}[Proof of (i)]
By Proposition \ref{prop:E(H,X)/Z(G)->Gamma}, $E(\{ H,X_{\alpha} \})$ meets every coset of $G$ in $E$ nontrivially, meaning $G$ and $E(\{ H,X_{\alpha} \})$ generate $E$. With this established, we see that there is a natural surjective homomorphism
\begin{center}
\begin{tabular}{c l l}
    $ G \rtimes E(\{ H,X_{\alpha} \})$ & $\rightarrow$ & $E$  \\
    $(g,e)$ & $\mapsto$ & $ge$. 
\end{tabular}
\end{center}
It is not hard to see that the kernel of this map is $\{ (z,z^{-1}) \mid z \in Z(G) \}$, i.e.,~the antidiagonal copy $Z(G)_{-\Delta}$. 
\end{proof}

\begin{subequations}\label{se:grpcoh}
Before we prove part (ii), we explain some details about the description of $E(\{H,X_\alpha\})$ mentioned in Theorem \ref{thm:listE}(i). The cohomology group
  $H^2(\Gamma,Z(G))$ can be described as the group
\begin{equation} Z^2(\Gamma,Z(G))/B^2(\Gamma,Z(G))\end{equation}
of cocycles modulo coboundaries. The set $Z^2(\Gamma,Z(G))$ of cocycles
consists of maps
\begin{equation}\begin{aligned} c&\, \colon \Gamma \times \Gamma
    \rightarrow Z(G),\\
    \gamma_1\cdot c(\gamma_2,\gamma_3) -c(\gamma_1\gamma_2,\gamma_3) &+
    c(\gamma_1,\gamma_2\gamma_3) - c(\gamma_1,\gamma_2) = 0\qquad
    (\gamma_i \in \Gamma).\end{aligned}\end{equation}
A coboundary begins with an arbitrary map $b\,\colon \Gamma \rightarrow
Z(G)$. The corresponding coboundary is
\begin{equation} db(\gamma_1,\gamma_2) = b(\gamma_1) + \gamma_1\cdot
  b(\gamma_2) -b(\gamma_1\gamma_2);
\end{equation}
  it is standard and easy to check that $db$ is a cocycle.
Given a cocycle $c$, the extension $E(\{H,X_\alpha\})$ is generated by
$Z(G)$ and additional elements
\begin{equation}
  \{\tilde\gamma \mid \gamma \in \Gamma\}
\end{equation}
which are subject to the relations
\begin{equation}\label{eq:discrel}
 \begin{aligned}
  \tilde\gamma_1 \tilde\gamma_2 &=
  c(\gamma_1,\gamma_2)\widetilde{\gamma_1\gamma_2}\\
  \tilde\gamma z \tilde\gamma^{-1} &= \gamma\cdot z \qquad (\gamma\in
  \Gamma, z \in Z(G)). \end{aligned}
\end{equation}

Conversely, if we are given an extension $E(\{H,X_\alpha\})$ as in
Theorem \ref{thm:listE}(i),
and we choose for each $\gamma\in \Gamma$ a preimage $\tilde\gamma \in
E(\{H,X_\alpha\})$, then these choices must
satisfy relations of the form \eqref{eq:discrel}, with $c$ a
cocycle. Replacing the representatives $\tilde\gamma$ with alternate
representatives
\begin{equation}
\tilde\gamma' = b(\gamma)\tilde\gamma \qquad (b(\gamma) \in Z(G))    
\end{equation}
(which are the only possibilities) replaces $c$ by $c' = c+db$.

More generally, consider a group $M$ and an abelian group $A$ together with a group action of $M$ on $A$. Recall that for $p \geq 0$, $H^p( M,A )$ can be described as the group
\begin{equation}
 Z^p( M,A )/B^p(M,A)     
\end{equation}
of cocycles modulo coboundaries. When $M$ is a finite group, the following result of Eckmann \cite{transfer} gives a condition under which the higher cohomology groups $H^p(M,A)$ (with $p > 0$) vanish. This result will help us prove Theorem \ref{thm:listE}(ii).  
\end{subequations} 

\begin{proposition}[Eckmann {\cite[Theorem 5]{transfer}}]\label{prop:transfer} 
Suppose that $\Gamma$ is a finite group of order $n$, and that $A$ is an abelian group on which $\Gamma$ acts. Then multiplication by $n$ acts by zero on $H^p(\Gamma,A)$. If multiplication by $n$ is an {\em automorphism} of $A$, then $H^p(\Gamma,A) = 0$ for all $p>0$.
\end{proposition}
\begin{proof}[Proof sketch] Because the functor
  $$H^0(\{1\},A) = A$$
  is exact, it follows that $H^p(\{1\},A) = 0$ for all $p>0$. Therefore the obvious restriction map
  $$Q\colon H^p(\Gamma,A) \rightarrow H^p(\{1\},A) \qquad (p>0)$$
  must be zero. Eckmann in \cite{transfer} defines a natural {\em transfer} homomorphism
  $$T\colon H^p(\{1\},A) \rightarrow H^p(\Gamma,A),$$
  and proves that $T\circ Q$ is multiplication by $n$. Since $Q=0$ for all $p>0$, it follows that multiplication by $n$ must be zero on $H^p(\Gamma,A)$ for all $p>0$. The last assertion is immediate. \end{proof}

Finally, we are ready to prove Theorem \ref{thm:listE}(ii).

\begin{proof}[Proof of (ii)]
Set $D:= Z(G)/Z(G)_\fin$ and consider the short exact sequence of $\Gamma$-modules
\begin{equation}\label{eq:ses}
  1\longrightarrow Z(G)_\fin \longrightarrow Z(G) \longrightarrow D \longrightarrow
1.\end{equation}
Because $Z(G)$ is a reductive abelian group, it is a direct sum of copies of ${\mathbb C}^\times$ and a finite abelian group; so $D$ is a direct sum of copies of ${\mathbb C}^\times/(\text{roots of unity})$. It follows easily that multiplication by $n$ is an automorphism of $D$ for every positive $n$. By Proposition \ref{prop:transfer},
$$H^p(\Gamma,D) = 0 \qquad (p > 0).$$
Examining the long exact sequence in $\Gamma$-cohomology attached to \eqref{eq:ses}, we deduce Theorem \ref{thm:listE}(ii). (In fact we can arrange for all values of a representative cocycle to have order dividing some power of $|\Gamma|$.)
\end{proof}

One reason that Theorem \ref{thm:listE}(ii) is interesting is for keeping calculations accessible to a computer. We would like the cocycle defining our disconnected group to take values in $Z(G)_\fin$, because
such elements are easily described in a computer. 

\begin{remark}
This argument works equally well over any algebraically closed field of characteristic zero (and shows that the extensions described by Theorem \ref{thm:listE} are independent of the field). In finite characteristic the same is true as long as the characteristic does not divide the order of $\Gamma$. If the characteristic {\em does} divide the order of $\Gamma$, then the extension $E$ is no longer a reductive group, and matters are more
complicated.
\end{remark}

\bibliography{biblio} 

\begin{thebibliography}{1}

\bibitem{AV}
J.~Adams and D.~A.~Vogan Jr.
\newblock Parameters for twisted representations.
\newblock In {\em Representations of Reductive Groups}, volume 312 of {\em Progress in Mathematics}, pages 51--116. Birkh\"{a}user, Cham, 2015.

\bibitem{CE}
H.~Cartan and S.~Eilenberg.
\newblock {\em Homological Algebra}.
\newblock Princeton University Press, Princeton, 1956.

\bibitem{Chev}
C.~Chevalley.
\newblock {\em Classification des Groupes Alg\'{e}briques Semi-simples}, volume~3 of {\em Collected Works}.
\newblock Springer-Verlag, Berlin, 2005.

\bibitem{transfer}
B.~Eckmann.
\newblock Cohomology of groups and transfer.
\newblock {\em Ann.~of Math.~(2)}, 58:481--493, 1953.

\bibitem{Springer}
T.~A. Springer.
\newblock {\em Linear Algebraic Groups}.
\newblock Modern Birkh\"{a}user Classics. Springer, Boston, 1998.

\end{thebibliography}
\bibliographystyle{plain}

\end{document}